\documentclass[reqno]{amsart}

\usepackage{amsmath}
\usepackage{amsfonts}
\usepackage{amssymb,enumerate}
\usepackage{amsthm}
\usepackage[all]{xy}
\usepackage{rotating}
\usepackage{hyperref}
\usepackage{color}

\theoremstyle{plain}
\newtheorem{lem}{Lemma}[section]
\newtheorem{cor}[lem]{Corollary}
\newtheorem{prop}[lem]{Proposition}
\newtheorem{thm}[lem]{Theorem}

\newtheorem*{mthm*}{Main Theorem}

\theoremstyle{definition}
\newtheorem{defn}[lem]{Definition}
\newtheorem{ex}[lem]{Example}

\newtheorem{disc}[lem]{Remark}

\newtheorem{para}[lem]{}

\newtheorem{notation}[lem]{Notation}

\newtheorem{convention}[lem]{Convention}
\newtheorem*{convention*}{Convention}

\newcommand{\id}{\operatorname{id}}

\newcommand{\HH}{\operatorname{H}}

\newcommand{\im}{\operatorname{Im}}
\newcommand{\shift}{\mathsf{\Sigma}}

\newcommand{\Ker}{\operatorname{Ker}}

\newcommand{\bbz}{\mathbb{Z}}

\newcommand{\xra}{\xrightarrow}

\newcommand{\e}{\mathbf{e}}

\renewcommand{\geq}{\geqslant}

\renewcommand{\ker}{\Ker}

\newcommand{\Ext}[4][R]{\operatorname{Ext}_{#1}^{#2}(#3,#4)}

\newcommand{\Hom}{\operatorname{Hom}}	
\newcommand{\Tor}[4][R]{\operatorname{Tor}^{#1}_{#2}(#3,#4)}

\def\Tor{\operatorname{Tor}}
\def\Ext{\operatorname{Ext}}

\def\e{\mathrm{e}}

\def\Diff{\mathrm{Diff}}

\def\C{\mathcal{C}}
\def\K{\mathcal{K}}
\def\Der{\mathrm{Der}}

\newcommand{\grHom}{\operatorname{gr-Hom}}

\numberwithin{equation}{lem}

\begin{document}

\bibliographystyle{amsplain}

\title[Obstruction to na\"{\i}ve liftability of DG modules]{Obstruction to na\"{\i}ve liftability of DG modules}

\author{Saeed Nasseh}
\address{Department of Mathematical Sciences\\
Georgia Southern University\\
Statesboro, GA 30460, U.S.A.}
\email{snasseh@georgiasouthern.edu}

\author{Maiko Ono}
\address{Institute for the Advancement of Higher Education, Okayama University of Science, Ridaicho, Kitaku, Okayama 700-0005, Japan}
\email{ono@ous.ac.jp}

\author{Yuji Yoshino}
\address{Graduate School of Natural Science and Technology, Okayama University, Okayama 700-8530, Japan}
\email{yoshino@math.okayama-u.ac.jp}

\thanks{Y. Yoshino was supported by JSPS Kakenhi Grant 19K03448.}



\begin{abstract}
The notion of na\"{\i}ve liftability of DG modules is introduced in~\cite{NOY} and~\cite{NOY2}. The main purpose of this paper is to explicitly describe the obstruction to na\"{\i}ve liftability along extensions $A\to B$ of DG algebras, where $B$ is projective as an underlying graded $A$-module. In particular, we give an explicit description of a DG $B$-module homomorphism which defines the obstruction to na\"{\i}ve liftability of a semifree DG $B$-module $N$ as a certain cohomology class in $\Ext^1_B(N,N\otimes_B J)$, where $J$ is the diagonal ideal. Our results on the obstruction class enable us to give concrete examples of DG modules that do and do not satisfy the na\"{\i}ve lifting property.
\end{abstract}

\maketitle


\section{Introduction}\label{sec20210306a}

Throughout the paper, $R$ is a commutative ring.\vspace{5pt}

Lifting and weak lifting properties have been studied in commutative ring theory by Auslander, Ding, and Solberg~\cite{auslander:lawlom} for modules and by Yoshino~\cite{yoshino} for complexes over certain ring extensions. These notions are tightly connected to the deformation theory of modules and have been applied in the theory of maximal Cohen-Macaulay approximations.
The existing lifting and weak lifting results for modules and complexes have been generalized recently in several papers~\cite{NOY, NOY2, nasseh:lql, nassehyoshino, OY} to the case of differential graded (DG) modules over certain DG algebra extensions.

There exist three versions of liftability for DG modules. Let $\varphi\colon A\to B$ be a DG $R$-algebra homomorphism, and let $N$ be a semifree DG $B$-module.
We say that $N$ is \emph{liftable} to $A$ if there is a semifree DG $A$-module $M$ such that $N\cong M\otimes_A B$. The DG $B$-module $N$ is \emph{weakly liftable} to $A$ (in the sense of~\cite{NOY}) if there are non-negative integers $a_1,\ldots,a_r$ such that the finite direct sum $N \oplus N(-a_1) \oplus \cdots \oplus N(-a_r)$ is liftable to $A$.
Finally, $N$ is {\it na\"ively liftable}  to $A$ if the map $\pi_N\colon N|_A\otimes_A B \to N$ defined by $\pi_N(n \otimes b) = nb$
is a split DG $B$-module epimorphism. Here, $N|_A$ denotes the DG $B$-module $N$ considered as a DG $A$-module via $\varphi$. If $N$  is na\"ively liftable to $A$, then it is a direct summand of the liftable DG $B$-module $N|_A\otimes_A B$; see Remark~\ref{para20210319a} for more details.
In the special case where $A$ and $B$ are commutative rings and $N$ is a $B$-module, definition of lifting (resp. weak lifting) translates to the existence of an $A$-module $M$ such that $N\cong M\otimes_A B$ (resp. $N$ is a direct summand of $M\otimes_A B$) and $\Tor_i^A(M,B)=0$ for all $i>0$.

Auslander, Ding, and Solberg~\cite[Proposition 1.6]{auslander:lawlom} proved a sufficient condition for liftability of modules over certain ring extensions. More precisely, they showed that
if $R$ is a complete local ring and $x\in R$ is a non-zero divisor, then a finitely generated $R/xR$-module $N$ is liftable to $R$ if $\Ext^2_{R/xR}(N,N)=0$.
The converse of this result does not hold in general; see~\cite[p. 282]{auslander:lawlom}. Unlike the lifting property, detecting weak liftability does not require the vanishing of the entire $\Ext^2$. More precisely, Auslander, Ding, and Solberg proved the following result.

\begin{thm}[\protect{\cite[Proposition 1.5]{auslander:lawlom}}]\label{thm20210821a'}
Let $R$ be a local ring and $x\in R$ be a non-zero divisor. A finitely generated $R/xR$-module $N$ is weakly liftable to $R$ if and only if a certain cohomology class
$[\Delta_N]$ in $\Ext^2_{R/xR}(N,N)$ vanishes.
\end{thm}

We call $[\Delta_N]$ the \emph{obstruction class} to weak liftability. Similar statements for lifting and weak lifting hold when
we replace modules by bounded below complexes of finitely generated free modules; see~\cite{yoshino}. Further generalizations have been obtained in the DG algebra setting for simple extensions of DG algebras as follows.


\begin{thm}[\protect{\cite[Theorem 3.6]{nassehyoshino} and ~\cite[Theorem 4.7]{OY}}]\label{thm20200605a}
Assume $B=A\langle X\rangle$ is a simple free extension of a DG $R$-algebra $A$ obtained by adjunction of a variable $X$ to kill a cycle $x$ in $A$. Let $N$ be a semifree DG $B$-module. Then there exists an obstruction class $[\Delta_N]$ in $\Ext^{|X|+1}_{B}(N,N)$ such that
\begin{enumerate}[\rm(a)]
\item
if $|X|$ is odd, then $[\Delta_N]=0$ if and only if $N$ is weakly liftable to $A$ and
\item
if $|X|$ is even and $N$ is bounded below (i.e., $N_i=0$ for all $i\ll 0$), then $[\Delta_N]=0$ if and only if $N$ is liftable to $A$.
\end{enumerate}
\end{thm}

Note that if $A=R$ is a local ring, $x$ is a non-zero divisor, and $|X|=1$, then Theorem~\ref{thm20200605a}(a) recovers Theorem~\ref{thm20210821a'} because in this case $B$ is the Koszul complex over the element $x$ which is quasiisomorphic, as a DG $R$-algebra, to $R/xR$.

In general, we are interested in the lifting property of DG modules along free extensions of DG algebras that are obtained by adjoining more than one variable to another DG algebra. Our study of lifting property along such extensions is motivated by some of the major problems in commutative algebra including the Auslander-Reiten Conjecture; see~\cite[$\S 7$]{NOY2} for more information. As we see in Theorem~\ref{thm20200605a}, weak liftablity along $A\to A\langle X\rangle$ depends on the parity of the degree of $X$. Hence, it is not a suitable version of liftability when one wishes to generalize the above result in an inductive process to the case where $B=A\langle X_i\mid i\in \mathbb{N}\rangle$. To avoid keeping track of the parity of the degrees of the variables in each inductive step, the notion of na\"{\i}ve liftability was introduced by the authors in~\cite{NOY, NOY2}. This notion is independent of the parity of the degrees of variables and at the same time, it detects (weak) liftability along simple extensions of DG algebras. Therefore, Theorem~\ref{thm20200605a} is restated as follows.

\begin{thm}[\protect{\cite[Theorem 6.8]{NOY}}]\label{cor20210822a}
Assume $B=A\langle X\rangle$ is a simple free extension of a DG $R$-algebra $A$ obtained by adjunction of a variable $X$ to kill a cycle in $A$. Let $N$ be a bounded below semifree DG $B$-module. Then there is an obstruction class $[\Delta_N]$ in $\Ext^{|X|+1}_{B}(N,N)$ such that $N$ is na\"{\i}vely liftable to $A$ if and only if
$[\Delta_N]=0$.
\end{thm}

Later, the authors proved the following \emph{existence} result along more general extensions of DG algebras using the notion of diagonal ideal that they introduced in~\cite{NOY2}.

\begin{thm}[\protect{\cite[Proposition 5.3 and Theorem 5.8]{NOY2}}]\label{cor20220321a}
Assume $A\to B$ is an extension of DG $R$-algebras, where $B$ is projective as an underlying graded $A$-module (e.g., $B$ is a free extension of the DG $R$-algebra $A$ obtained by adjoining countably many variables to $A$). For a semifree DG $B$-module $N$, there is an obstruction class $[\Delta_N]$ in $\Ext^{1}_B(N,N\otimes_B J)$ such that $N$ is na\"{i}vely liftable to $A$ if and only if $[\Delta_N]=0$. (Here, $J$ is the diagonal ideal; see Definition~\ref{para20210801a}.)
\end{thm}

The obstruction classes $[\Delta_N]$ discussed in Theorems~\ref{thm20210821a'}, \ref{thm20200605a}, and~\ref{cor20210822a} have been \emph{explicitly described}  in~\cite{auslander:lawlom}, \cite{nassehyoshino}, and~\cite{OY}.
Our goal in this paper is to \emph{explicitly describe} the obstruction to na\"{\i}ve liftability discussed in Theorem~\ref{cor20220321a} as well. Our main result in this paper is the following; see Theorems~\ref{prop20210802c} and~\ref{main}.

\begin{thm}\label{thm20210822a}
Let $A\to B$ be an extension of DG $R$-algebras, where $B$ is projective as an underlying graded $A$-module, and let $N$ be a semifree DG $B$-module with a semifree basis $\mathcal{B}=\{e_{\lambda}\}_{\lambda\in \Lambda}$. For $e_{\lambda}\in\mathcal B$, let $\partial^N(e_{\lambda})=\sum_{\mu < \lambda} e_{\mu}b_{\mu\lambda}$ as a finite sum with $b_{\mu\lambda}\in B^{\natural}$. Then
a right DG $B$-module homomorphism $\Delta_N\colon N\to N\otimes_B J$ of degree $-1$ which defines the obstruction class $[\Delta_N]$ of na\"{i}ve liftability is explicitly given by the formula
$$
\Delta_N(e_{\lambda})=\sum_{\mu < \lambda}e_{\mu}\otimes \delta(b_{\mu\lambda})
$$
where $\delta$ is the universal derivation; see Definition~\ref{para20210317c}.
\end{thm}

The organization of this paper is as follows. Section~\ref{sec20210306b} is devoted to the terminology
and basic definitions which are used in subsequent sections. In Section~\ref{sec20210318a} we give an explicit description of the obstruction class whose vanishing detects na\"{\i}ve liftability. The proof of Theorem~\ref{thm20210822a} is given in this section. In Section~\ref{sec20210822a} we provide another description of the obstruction class that is equivalent to the one from Section~\ref{sec20210318a}. This description is based on the notion of ``connections'' that was originally defined by Connes~\cite[II. \S 2]{AC} in non-commutative differential geometry. Finally, our main result in Section~\ref{sec20220306a} is Theorem~\ref{20210626b} which enables us to construct concrete examples of DG modules that do and do not satisfy na\"{\i}ve liftability.

\section{Terminology and preliminaries}\label{sec20210306b}

The main objects considered in this paper are DG algebras and DG modules; references on these subjects include~\cite{avramov:ifr,avramov:dgha, felix:rht, GL}. In this section, we fix our notation and specify some terminology that will be used throughout the paper. For the unspecified notation, we refer the reader to~\cite{NOY2}.

\begin{para}\label{para20200329a}
A \emph{strongly commutative differential graded $R$-algebra} (or simply \emph{DG $R$-algebra}) $A=(A^{\natural},d^A)$ is a non-negatively graded $R$-algebra $A^{\natural}  = \bigoplus  _{n \geq 0} A _n$ such that
\begin{enumerate}[\rm(a)]
\item
$d^A\colon A^{\natural}\to A^{\natural}$ is a graded $R$-linear map  of degree $-1$ with $(d^A)^2=0$, that is, $A=(A^{\natural},d^A)$ is an $R$-complex;
\item
for all homogeneous elements $a, b \in A^{\natural}$ we have $ab = (-1)^{|a| |b|}ba$, and  $a^2 =0$  if the degree of $a$ (denoted $|a|$) is odd; and
\item
$d^A$ satisfies the \emph{Leibniz rule}, that is, for all homogeneous elements $a,b\in A^{\natural}$ we have $d^A(ab) = d^A(a) b + (-1)^{|a|}ad^A(b)$.
\end{enumerate}
\end{para}

\begin{para}\label{para20201112c}
A (right) \emph{DG $A$-module} $M=(M^{\natural},\partial^M)$ is a graded (right) $A^{\natural}$-module $M^{\natural}=\bigoplus_{i\in \mathbb{Z}}M_i$ which is an $R$-complex with a differential $\partial^M$ that satisfies the Leibniz rule, i.e., for all homogeneous elements $a\in A^{\natural}$ and $m\in M^{\natural}$ we have
$\partial^M(ma) = \partial^M(m)\ a + (-1)^{|m|} m\ d^A(a)$.
A \emph{DG submodule} of a DG $A$-module $M$ is a subcomplex that is a DG $A$-module under the operations induced by $M$. A \emph{DG ideal} of $A$ is a right and left DG submodule of $A$.

DG modules considered in this paper are right DG modules, unless otherwise stated. Note that a DG $A$-module $M$ is also a left DG $A$-module with the left $A$-action
$am = (-1)^{|m||a|} ma$,  for all homogeneous elements $a\in A^{\natural}$ and $m \in M^{\natural}$.

A \emph{DG $A$-module homomorphism} between DG
$A$-modules $M, N$ is an $A^{\natural}$-linear map $f\colon M^{\natural}\to N^{\natural}$ that \emph{commutes with differentials}, i.e., $\partial^Nf=f\partial^M$.
We also set
$$
{}^{*}\!\Hom_{A^{\natural}}(M^{\natural},N^{\natural})
= \bigoplus _{n \in \bbz} \grHom_{A^{\natural}}(M^{\natural},N^{\natural}(n))
$$
where $\grHom_{A^{\natural}}(M^{\natural},N^{\natural}(n))$ is the set of graded $A^{\natural}$-module homomorphisms from $M^{\natural}$ to $N^{\natural}$ of degree $n$. We denote $\grHom_{A^{\natural}}(M^{\natural}, N^{\natural}(0))$ by $\grHom_{A^{\natural}}(M^{\natural}, N^{\natural})$. Note that ${}^{*}\!\Hom_{A^{\natural}}(M^{\natural},N^{\natural})$ is a graded $A^{\natural}$-module. Defining the differential $\partial^{*}$ by
$$(\partial^{*} f) (x) = \partial ^{N} (f(x)) -(-1)^{|f|}f(\partial ^{M}(x))$$
for $f\in {}^{*}\!\Hom_{A^{\natural}}(M^{\natural},N^{\natural})$ and $x\in M^{\natural}$, we see that $\left({}^{*}\!\Hom_{A^{\natural}}(M^{\natural},N^{\natural}),
\partial^{*}\right)$ has a DG $A$-module structure, which we denote by $\Hom_A(M,N)$.

For a DG $A$-module $M$ and an integer $i$, the \emph{$i$-th shift} of $M$, denoted $\shift^i M$, is defined by $\left(\shift^i M\right)_j = M_{j-i}$ with $\partial_j^{\shift^i M}=(-1)^i\partial_{j-i}^M$. We set $\shift M=\shift^1 M$.

A DG $A$-module $M$ is \emph{semifree} if it has a \emph{semifree basis} $\{e_{\lambda}\}_{\lambda\in\Lambda}$, that is, a subset $\{e_{\lambda}\}_{\lambda\in\Lambda}\subseteq M^{\natural}$ which is a graded $A^{\natural}$-free basis of $M^{\natural}$ indexed by a well-orderd set
$(\Lambda, <)$ such that $\partial^M(e_\lambda)=\sum_{\mu<\lambda} e_\mu A^{\natural}$ for all $e_\lambda$.
See~\cite[8.2]{avramov:dgha} for more details.

For an integer $i$ and DG $A$-modules $M,N$, where $M$ is semifree, $\Ext^i_A(M,N)$ is defined to be $\HH_{-i}\left(\Hom_A(M,N)\right)$.
\end{para}

\begin{para}\label{para20210423b}
Let $\C(A)$ denote the abelian category of DG $A$-modules and DG $A$-module homomorphisms.
Let $\K(A)$ be the homotopy category of DG $A$-modules; objects of $\K(A)$ are DG $A$-modules and morphisms are the set of homotopy equivalence classes of DG $A$-module homomorphisms, i.e., $\Hom _{\K(A)} (M, N) = \Hom _{A} (M, N)/ \sim$,
where $f \sim g$ for $f,g\in \Hom _{A} (M, N)$ if and only if there is a graded $A^{\natural}$-module homomorphism $h\colon M^{\natural} \to N^{\natural}$ of degree $-1$ such that $f - g = \partial ^N h + h \partial ^M$. The category $\K(A)$ is a triangulated category
and $\Ext^i_A(M,N)=   \Hom _{\K(A)}(M, \shift^i N)$.
\end{para}

\begin{para}\label{para20201114e}
Assume that $A, B$ are DG $R$-algebras such that $B^{\natural}$ is projective as a graded $A^{\natural}$-module. Let $B^o$ be the opposite DG $R$-algebra and $B^e=B^o \otimes_A B$ be the enveloping DG $R$-algebra of $B$ over $A$. The algebra structure on $B^e$ is given by
\begin{equation}\label{eq20210317d}
 (b_1^o \otimes b_2)( {b'}_1^o \otimes {b'}_2)
=  (-1)^{|{b'}_1|\left(|b_1|+|b_2|\right)}({b'}_1 b_1)^o\otimes b_2{b'}_2
\end{equation}
and its differential structure is described as
\begin{equation}\label{eq20210318a}
d^{B^e} (  b_1^o \otimes b_2 ) = d^{B^o} (  b_1^o) \otimes b_2 + (-1)^{|b_1|} b_1^o \otimes d^{B}(b_2 )
\end{equation}
for all homogeneous elements $b_1, b_2, b'_1, b'_2\in B^{\natural}$.

Note that DG $B^e$-modules are precisely DG $(B, B)$-bimodules. In fact, for a DG $B^e$-module $N$, the right action of an element of $(B^e)^{\natural}$ on $N^{\natural}$ yields the two-sided $B^{\natural}$-module structure
\begin{equation}\label{eq20210317a}
nb=n(1 ^o \otimes b)\qquad\text{and}\qquad bn=(-1)^{|b||n|}n(b^o\otimes 1)
\end{equation}
for all homogeneous elements $n \in N^{\natural}$ and $b \in B^{\natural}$. In particular, if $N=B^e$ and $n=b_1^o\otimes b_2$ for $b_1, b_2\in B^{\natural}$, then by~\eqref{eq20210317d} we have
\begin{equation}\label{eq20210523a}
(b_1 ^o \otimes b_2)b=b_1 ^o \otimes b_2b\qquad\text{and}\qquad b(b_1 ^o \otimes b_2)=(bb_1)^o\otimes b_2.
\end{equation}
\end{para}

\begin{para}\label{para20210523a}
Assume that $A, B$ are DG $R$-algebras such that $B^{\natural}$ is projective as a graded $A^{\natural}$-module. For a DG $B^e$-module $N$, a graded $A^{\natural}$-linear map $D\colon B^{\natural}\to N^{\natural}$ is called an \emph{$A$-derivation of $N$} if the following conditions are satisfied for all homogeneous elements $n \in N^{\natural}$ and $b_1, b_2 \in B^{\natural}$:
\begin{gather}
D(b_1b_2)=D(b_1)b_2+(-1)^{|D||b_1|}b_1D(b_2)\label{eq20210317b}\\
D(a)=0\  \text{for all $a\in A$}.\notag
\end{gather}
Note that by~\eqref{eq20210317a}, equality~\eqref{eq20210317b} is equivalent to
\begin{equation}\label{eq20210317c}
D(b_1b_2)=D(b_1)\left(1^o\otimes b_2\right)+(-1)^{|b_1||b_2|}D(b_2)\left(b_1^o\otimes 1\right).
\end{equation}
We denote the set of $A$-derivations of the DG $B^e$-module $N$ by $\Der_A(B,N)$.
\end{para}

\section{Explicit description of the obstruction class}\label{sec20210318a}

Our main results in this section are Theorems~\ref{prop20210802c} and~\ref{main} in which we explicitly describe the obstruction class to na\"{\i}ve liftability of DG modules along DG algebra extensions. The notation used in this section comes from Section~\ref{sec20210306b}. We also make the following convention for the rest of the paper.

\begin{convention}\label{para20210317a}
Throughout the paper, $A, B$ are DG $R$-algebras such that $B^{\natural}$ is projective as a graded $A^{\natural}$-module. Also, $\varphi\colon A\to B$ is a \emph{DG $R$-algebra homomorphism}, that is, $\varphi$ is a graded $R$-algebra homomorphism of degree $0$ with $d^B\varphi=\varphi d^A$.
\end{convention}

The description of the obstruction class to na\"{\i}ve liftability of DG modules is based on the notion of diagonal ideal that we define next.

\begin{defn}\label{para20210801a}
Let $\pi\colon B^e \to B$ denote the map defined by $\pi(b_1 ^o \otimes b_2) = b_1b_2$, for all $b_1, b_2\in B^{\natural}$. Note that $\pi$ is a homomorphism of DG $R$-algebras; see~\cite[3.1]{NOY2}. Also,
$J:=\ker \pi$ is a DG ideal of $B^e$ and following~\cite{NOY2}, it is called the \emph{diagonal ideal} of $\varphi$.
Moreover, the isomorphism $B^e /J \cong B$ of DG $R$-algebras is also an isomorphism of DG $B^e$-modules. Hence, there is an exact sequence
\begin{equation}\label{eq20201114a}
0 \to J \xra{\iota}  B^e \xra{\pi} B  \to 0
\end{equation}
of DG $B^e$-modules in which $\iota$ is the natural inclusion.
\end{defn}

\begin{prop}\label{prop20210318b}
The short exact sequence~\eqref{eq20201114a} is a splitting sequence of DG $B$-modules. More precisely, the maps $\rho\colon B\to B^e$ and $\sigma\colon B^e\to J$  defined by
\begin{align*}
\rho(b)=1^o\otimes b&\qquad\qquad\text{and}\qquad\qquad
\sigma(b_1^o\otimes b_2)=b_1^o\otimes b_2-1^o\otimes b_1b_2
\end{align*}
for all $b, b_1, b_2\in B^{\natural}$, are DG $B$-module homomorphisms that satisfy the equalities
\begin{equation}\label{eq20210802a}
\pi\rho=\id_{B},\qquad\qquad\sigma\iota=\id_{J},\qquad\qquad\iota\sigma+\rho\pi=\id_{B^e}.
\end{equation}
\end{prop}

\begin{proof}
For all $b_1, b_2\in B^{\natural}$ we have
$\pi\sigma(b_1^o\otimes b_2)=\pi\left(b_1^o\otimes b_2-1^o\otimes b_1b_2\right)=b_1b_2-b_1b_2=0$.
Hence, $\im(\sigma)\subseteq J=\ker \pi$, that is, the map $\sigma$ is well-defined.

Note that the short exact sequence~\eqref{eq20201114a} is a splitting sequence of graded $B^{\natural}$-modules and the $B^{\natural}$-module homomorphism $\rho$ satisfies the equality $\pi\rho=\id_{B}$. Moreover, it follows from~\eqref{eq20210318a} that $\rho$ commutes with differentials. Hence, $\rho$ is a DG $B$-module homomorphism.

It is straightforward to check the equality $\iota\sigma+\rho\pi=\id_{B^e}$. To prove the equality $\sigma\iota=\id_{J}$, note that the elements of $J^{\natural}$ are finite sums of the form $\sum_{i=1}^n b_i^o\otimes b'_i$ with homogeneous elements $b_i,b'_i\in B^{\natural}$ such that $\sum_{i=1}^nb_ib'_i=\sum_{i=1}^n \pi(b_i^o\otimes b'_i)=\pi\left(\sum_{i=1}^n b_i^o\otimes b'_i\right)=0$. For such an element in $J^{\natural}$ we have
$$
\sigma \iota\left(\sum_{i=1}^n b_i^o\otimes b'_i\right)=\sum_{i=1}^n (b_i^o\otimes b'_i-1^o\otimes b_ib'_i)=\sum_{i=1}^n b_i^o\otimes b'_i.
$$

Also, $\sigma$ commutes with the differentials. In fact, for all homogeneous elements $b_1,b_2\in B^{\natural}$ we have the equalities
\begin{align*}
\partial^{J}\sigma(b_1^o\otimes b_2)&=\partial^{J}(b_1^o\otimes b_2-1^o\otimes b_1b_2)\\
&=d^{B^o} (b_1^o)\otimes b_2+(-1)^{|b_1|}b_1^o\otimes d^B(b_2)-1^o\otimes d^B(b_1b_2)\\
&=d^{B^o} (b_1^o)\otimes b_2-1^o\otimes d^B(b_1)b_2+(-1)^{|b_1|}\left(b_1^o\otimes d^B(b_2)-1^o\otimes b_1d^B(b_2)\right)\\
&=\sigma(d^{B^o} (  b_1^o) \otimes b_2 + (-1)^{|b_1|} b_1^o \otimes d^{B}(b_2))\\
&=\sigma d^{B^e}(b_1^o\otimes b_2).
\end{align*}
This computation implies that $\sigma\colon B^e\to J$ is a DG $B$-module homomorphism.
\end{proof}




\begin{disc}\label{para20210319a}
Let $N$ be a semifree DG $B$-module. Applying the functor $N\otimes_B -$ to the short exact sequence~\eqref{eq20201114a} and using the natural isomorphism $\nu\colon N\otimes_BB\to N$ of DG $B$-modules defined by $\nu(n\otimes b)=nb$ for all $n\in N^{\natural}$ and $b\in B^{\natural}$, we obtain the short exact sequence at the top row of the commutative diagram
\begin{equation}\label{eq20210319a}
\xymatrix{
0\ar[r]&N\otimes_B J\ar[r]^{\iota_N}&N\otimes_B B^e\ar[rrr]^{\nu(\id_N\otimes \pi)}\ar[d]_{\cong}^{\ell}&&&N\ar[r]\ar@{=}[d]&0\\
&&N|_A\otimes_AB\ar[rrr]^{\pi_N}&&&N&
}
\end{equation}
in which $\iota_N=\id_N\otimes \iota$, the DG $B$-module $N$ regarded as a DG $A$-module via the DG $R$-algebra homomorphism $\varphi$ is denoted by $N |_A$, the map $\pi_N$ is the DG $B$-module epimorphism
defined by $\pi_N(n \otimes b)=nb$, and the isomorphism $\ell$ is the composition of the isomorphisms $N\otimes_B B^e= N\otimes_B(B^o\otimes_AB)\cong (N\otimes_BB^o)\otimes_AB\cong N|_A\otimes_AB$ which is described by the formula $\ell(n\otimes(b_1^o\otimes b_2))= nb_1\otimes b_2$ for all $n\in N^{\natural}$ and $b_1, b_2\in B^{\natural}$.
Using the commutative diagram~\eqref{eq20210319a}, we identify the map $\nu(\id_N\otimes \pi)$ by $\pi_N$, i.e., we assume $\pi_N=\nu(\id_N\otimes \pi)$ for the rest of the paper and hence, the top row of~\eqref{eq20210319a} is the short exact sequence
\begin{equation}\label{eq20210802b}
0 \to N \otimes _BJ \xra{\iota_N} N\otimes_B B^e \xra{\pi_N} N  \to 0
\end{equation}
of DG $B$-modules; see~\cite[Proposition 5.3]{NOY2} for more details.
Note that $N$  is na\"ively liftable to $A$ if and only if ~\eqref{eq20210802b} splits. In particular, if $N$ is na\"ively liftable to $A$, then it is a direct summand of the DG $B$-module $N |_A \otimes _A B$, which is liftable to $A$.
\end{disc}

\begin{prop}\label{prop20210802a}
Let $N$ be a semifree DG $B$-module with a semifree basis $\mathcal{B}=\{e_{\lambda}\}$. Then~\eqref{eq20210802b} is a splitting sequence of graded $B^{\natural}$-modules. More precisely, the maps $\rho_N\colon N^{\natural} \to (N\otimes_BB^e)^{\natural}$ and $\sigma_N\colon (N\otimes_BB^e)^{\natural}\to (N\otimes_BJ)^{\natural}$ defined by
\begin{align*}
\rho_N(e_{\lambda} b)=e_{\lambda}\otimes \rho(b)&\qquad\qquad\text{and}\qquad\qquad
\sigma_N(e_{\lambda}\otimes (b_1^o\otimes b_2))=e_{\lambda}\otimes\sigma(b_1^o\otimes b_2)
\end{align*}
for all $e_{\lambda}\in \mathcal B$ and all homogeneous elements $b,b_1,b_2\in B^{\natural}$, are right $B^{\natural}$-linear homomorphisms for which the following equalities hold: \begin{equation}\label{eq20210524a}
\pi_N\rho_N=\id_{N},\qquad \sigma_N\iota_N=\id_{N\otimes_B J},\qquad \iota_N\sigma_N+\rho_N\pi_N=\id_{N\otimes_BB^e}.
\end{equation}
\end{prop}

\begin{proof}
For all $e_\lambda\in \mathcal{B}$ and $b\in B^{\natural}$ we see that
\begin{gather}
\pi_N\rho_N(e_\lambda)=\pi_N(e_\lambda \otimes \rho(1))=e_\lambda\pi\rho(1)=e_\lambda=\id_N(e_\lambda)\notag\\
\left(\iota_N\sigma_N+\rho_N\pi_N\right)(e_\lambda \otimes (b^o \otimes 1))
=e_\lambda\otimes\sigma(b^o\otimes 1)+e_\lambda\otimes \rho(b)=e_\lambda \otimes (b^o \otimes 1).\notag
\end{gather}
Therefore, the first and third equalities of~\eqref{eq20210524a} hold since $\pi_N,\,\rho_N,\,\sigma_N$, and $\iota_N$ are $B^{\natural}$-linear.
Now, let $e_\lambda\in\mathcal{B}$ and $\sum_i b_i\otimes b'_i\in J$.
Then we have
\begin{align*}
\sigma_N\iota_N\left(e_\lambda\otimes\sum_i b_i\otimes b'_i\right)
=e_\lambda\otimes \sigma\iota\left(\sum_i b_i\otimes b'_i\right)
=e_\lambda\otimes \sum_i b_i\otimes b'_i.
\end{align*}
Hence, the second equality of~\eqref{eq20210524a} holds.
\end{proof}

\begin{disc}
Note that $\rho_N$ and $\sigma_N$ described in Proposition~\ref{prop20210802a} are not DG $B$-module homomorphisms in general.
\end{disc}

Next definition plays an essential role in the rest of the paper.

\begin{defn}\label{para20210317c}
Let $\delta\colon B^{\natural}\to J^{\natural}$ be the $A^{\natural}$-linear map defined by $$\delta(b)=b^o\otimes 1-1^o\otimes b$$ for all $b\in B^{\natural}$. Note that $|\delta|=0$ and $\delta(a)=0$ for all $a\in A^{\natural}$. Also, it follows from~\eqref{eq20210317d} that $\delta$ satisfies~\eqref{eq20210317c}. Hence, $\delta\in \Der_A(B,J)$. We call $\delta$ the \emph{universal derivation}.
Since $\partial^J $ is induced from $d^{B^e}$, the universal derivation $\delta$ commutes with the differentials, i.e.,
\begin{equation}\label{eq20210523s}
\delta d^B=\partial^J \delta.
\end{equation}
\end{defn}

\begin{disc}\label{disc20220321a}
Let $N$ be a semifree DG $B$-module, and let $\Delta_N\colon N^{\natural}\to (N\otimes_BJ)^{\natural}$ be the right $B^{\natural}$-linear graded homomorphism of degree $-1$ defined by
$$\Delta_N:=\sigma_N\partial^{N\otimes_BB^{e}}\rho_N.$$
It follows from chasing the diagram
$$
\xymatrix@C=7mm@R=4mm{
&& N\otimes_B B^e \ar@<1ex>[rr]^-{\pi_N} \ar[dd]_{\partial^{N\otimes_B B^e}}
&& N  \ar@{.>}[ll]^-{\rho_N}  \ar[dd]^-{\partial^N} \\
&&&&\\
N\otimes_B J   \ar@<1ex>[rr]^-{\iota_N} \ar[dd]_{\partial^{N\otimes_B J}}
&& N\otimes_B B^e \ar@<1ex>[rr]^-{\pi_N} \ar[dd]^{\partial^{N\otimes_B B^e}} \ar@{.>}[ll]^-{\sigma_N}
&& N  \ar@{.>}[ll]^-{\rho_N}   \\
&&&&\\
N\otimes_B J   \ar@<1ex>[rr]^-{\iota_N}
&& N\otimes_B B^e  \ar@{.>}[ll]^-{\sigma_N}
&&}
$$
that $\iota_N(\partial^{N\otimes_B J} \Delta_N - \Delta_N \partial^{N})=0$.
Therefore, $\Delta_N$ is a DG $B$-module homomorphism, since $\iota_N$ is injective.
In particular, $\Delta_N$ defines a cohomology class in $\Ext_B^1(N,N\otimes_B J)$ which is denoted by $[\Delta_N]$.
\end{disc}

The next theorem is one of the main results in this section that describes the structure of the obstruction class to na\"{\i}ve liftability of DG modules along the DG algebra extensions satisfying Convention~\ref{para20210317a}. This result  is a major part of Theorem~\ref{thm20210822a} from the introduction.

\begin{thm}\label{prop20210802c}
Let $N$ be a semifree DG $B$-module with a semifree basis $\mathcal{B}=\{e_{\lambda}\}_{\lambda\in \Lambda}$.
 For $e_{\lambda}\in\mathcal B$, if we assume $\partial^N(e_{\lambda})=\sum_{\mu<\lambda} e_{\mu}b_{\mu\lambda}$ as a finite sum with $b_{\mu\lambda}\in B^{\natural}$, then
the DG $B$-module homomorphism $\Delta_N\colon N\to N\otimes_B J$ is explicitly described by the formula
\begin{equation}\label{eq20210523b}
\Delta_N(e_{\lambda})=\sum_{\mu<\lambda}e_{\mu}\otimes \delta(b_{\mu\lambda}).
\end{equation}
\end{thm}
\begin{proof}
For $e_\lambda \in \mathcal{B}$ we have the following equalities:
\begin{equation*}
\begin{split}
\Delta_N(e_\lambda)
&= \sigma_N\partial^{N\otimes_B B^e}\rho_N (e_\lambda)
 =\sigma_N\partial^{N\otimes_B B^e} (e_\lambda \otimes (1^o \otimes 1))\\
 &= \sigma_N\left (\sum_{ \mu <\lambda } e_\mu \otimes (b_{\lambda \mu}^o \otimes 1)\right)
 = \sum_{\mu<\lambda} e_\mu \otimes \delta ( b_{\lambda \mu})
\end{split}
\end{equation*}
as desired.
\end{proof}

\begin{disc}\label{para20210524a}
Consider the notation from Theorem~\ref{prop20210802c}.
Since~\eqref{eq20201114a} is a split short exact sequence of graded $B^{\natural}$-modules, we have the isomorphism $$N^{\natural}\otimes_{B^{\natural}} (B^e)^{\natural}\cong (N^{\natural}\otimes_{B^{\natural}} J^{\natural})\oplus N^{\natural}.$$ By Proposition~\ref{prop20210802a}, the sequence~\eqref{eq20210802b} is equivalent to a short exact sequence
\begin{equation}\label{eq20210524b}
0\to N\otimes_BJ\xra{{\tiny \left(\begin{matrix}\id_{N\otimes_BJ}\\0\end{matrix}\right)}}(N\otimes_BJ)\oplus N\xra{{\tiny \left(\begin{matrix}0&\id_{N}\end{matrix}\right)}}N\to 0
\end{equation}
of DG $B$-modules in which the differential on $(N\otimes_BJ)\oplus N$ is of the form
\begin{equation}\label{eq20210524c}
\partial:=\left(\begin{matrix}
\partial^{N\otimes_BJ} & f_N \\
0 & \partial^N
\end{matrix}\right)
\end{equation}
where $f_N\colon N\to N\otimes_BJ$ is a cycle in $\Hom_B(N,\shift (N\otimes_BJ))$. Hence, $f_N$ represents a cohomology class $[f_N]$ in $\Ext^1_B(N,N\otimes_BJ)$. Note also that
$$f_N=\left(\begin{matrix}\id_{N\otimes_BJ}&0\end{matrix}\right)\partial\left(\begin{matrix}0\\\id_{N}\end{matrix}\right).$$
Therefore, $[\Delta_N]=[f_N]$ in $\Ext^1_B(N,N\otimes_BJ)$.
\end{disc}


The following along with Theorem~\ref{prop20210802c} completes the proof of Theorem~\ref{thm20210822a}.

\begin{thm}\label{main}
Consider the notation from Theorem~\ref{prop20210802c} and Remark~\ref{para20210524a}. The following conditions are equivalent:
\begin{enumerate}[\rm(i)]
\item
$[f_N]=[\Delta_N]=0$ in $\Ext^1_B(N,N\otimes_BJ)$;
\item
$N$ is na\"{\i}vely liftable to $A$.
\end{enumerate}
\end{thm}

\begin{proof}
Let $\widetilde{\partial}:=\left(\begin{smallmatrix}\partial^{N\otimes_BJ}&0\\0&\partial^N\end{smallmatrix}\right)$ and recall from~\eqref{eq20210524c} that
$\partial=\left(\begin{smallmatrix}
\partial^{N\otimes_BJ} & f_N \\
0 & \partial^N
\end{smallmatrix}\right)$. When necessary, we use the notations $\left((N\otimes_{B}J)^\natural\oplus N^\natural, \partial\right)$ and $\left((N\otimes_{B}J)^\natural\oplus N^\natural, \widetilde{\partial}\right)$ to specify the DG $B$-module structure on $(N\otimes_{B}J)^\natural\oplus N^\natural$.

Note that
$N$ is naively liftable to $A$ if and only if
there is an isomorphism $$\Phi\colon ((N\otimes_{B}J)^\natural\oplus N^\natural, \partial) \to  ((N\otimes_{B}J)^\natural\oplus N^\natural, \widetilde{\partial})$$
of DG $B$-modules such that the diagram
\begin{equation}
\begin{split}
\xymatrix@C=15mm@R=8mm{
0\ar[r]
&N\otimes_BJ\ar@{=}[d]\ar[r]^-{{\tiny \left(\begin{matrix}\id_{N\otimes_BJ}\\0\end{matrix}\right)}}
&\left((N\otimes_{B}J)^\natural\oplus N^\natural, \partial\right)
\ar[d]_{\Phi}\ar[r]^-{{\tiny \left(\begin{matrix}0&\id_{N}\end{matrix}\right)}}
&N\ar@{=}[d]\ar[r]
&0
\\
0\ar[r]
&N\otimes_BJ\ar[r]^-{{\tiny \left(\begin{matrix}\id_{N\otimes_BJ}\\0\end{matrix}\right)}}
&\left((N\otimes_{B}J)^\natural\oplus N^\natural, \widetilde{\partial}\right)
\ar[r]^-{{\tiny \left(\begin{matrix}0&\id_{N}\end{matrix}\right)}}
&N\ar[r]
&0
}\notag
\end{split}
\end{equation}
commutes.
Commutativity of this diagram implies that $\Phi$ is of the form
$$
 \Phi=\left(\begin{matrix}
\id_{N\otimes_BJ} & q \\
0 & \id_{N}
\end{matrix}\right)
$$
where $q\colon N^\natural \to (N\otimes_BJ)^\natural$ is a graded homomorphism of degree $0$.

(i)$\implies$(ii) Since $[f_N]=0$ in $\Ext^1_B(N,N\otimes_BJ)$, there exists a graded $B^{\natural}$-module homomorphism $g\colon N^{\natural}\to (N\otimes_BJ)^{\natural}$ of degree $0$ such that $f_N=\partial^{N\otimes_BJ}g-g\partial^N$.
Setting $q:=g$ in $\Phi$, we see that $N$ is na\"{i}vely liftable to $A$.

(ii)$\implies$(i) If $N$ is na\"{\i}vely liftable to $A$, then an isomorphism
$$
\Phi=\left(\begin{matrix}
\id_{N\otimes_BJ} & q \\
0 & \id_{N}
\end{matrix}\right)\colon ((N\otimes_{B}J)^\natural\oplus N^\natural, \partial) \to  ((N\otimes_{B}J)^\natural\oplus N^\natural, \widetilde{\partial})$$
 of DG $B$-modules exists.
Since $\Phi$ is a DG $B$-module homomorphism, it commutes with the differentials and $q$ is a DG $B$-module homomorphism, that is, $q\in\Hom_B(N,N\otimes_BJ)$. Hence, $f_N=\partial^{N\otimes_BJ}q-q\partial^N$. This means that $[f_N]=0$ in $\Ext^1_B(N,N\otimes_BJ)$, as desired.
\end{proof}

\begin{defn}\label{para20210804a}
Following Theorems~\ref{prop20210802c} and~\ref{main}, for a semifree DG $B$-module $N$, we call $[f_N]=[\Delta_N]$ the \emph{obstruction class} to na\"{\i}ve liftability.
\end{defn}

\section{Another description of the obstruction class}\label{sec20210822a}

In this section, we provide another description of the obstruction class to na\"{\i}ve liftability that is equivalent to the one constructed in Theorem~\ref{prop20210802c}; see Theorem~\ref{20210706d} below, which is our main result in this section.
The notation used in this section comes from the previous sections. Recall that we still work in the setting of Convention~\ref{para20210317a}.
We start with introducing the notion of ``connections'' that was originally defined by Connes~\cite[II. \S 2]{AC} in non-commutative differential geometry. See also Cuntz and Quillen~\cite[\S 8]{CQ}.

\begin{defn}\label{para20210804c}
Let $N$ be a semifree DG $B$-module with a semifree basis $\mathcal{B}=\{e_{\lambda}\}_{\lambda\in \Lambda}$. We define a subset $\Diff_A^\delta(N)$ of $\grHom_{A^{\natural}}(N^{\natural},(N\otimes_BJ)^{\natural})$ by
$$
\Diff_A^\delta(N) = \{ D\colon N^{\natural}\to (N\otimes_BJ)^{\natural}\mid D(nb)=D(n)b+n\otimes \delta(b)\ \text{for}\ n\in N^{\natural}, b\in B^{\natural}\}.
$$
Each element of $\Diff_A^\delta(N)$ is called a {\it connection on $N$}.

Let $D^{\mathcal{B}}\colon N^{\natural}\to (N\otimes_BJ)^{\natural}$ be the graded $A^{\natural}$-linear homomorphism of degree $0$ satisfying $D^{\mathcal{B}}(\sum_{\lambda}e_\lambda b_\lambda)=\sum_{\lambda}e_\lambda \otimes\delta(b_\lambda)$, for all $e_\lambda\in \mathcal{B}$ and $b_\lambda\in B$.
Noting that $D^{\mathcal{B}}(e_{\lambda})=0$ for all $e_\lambda\in\mathcal{B}$, we have $D^{\mathcal{B}}\in \Diff_A^\delta(N)$.
Thus, $\Diff_A^\delta(N)$ is non-empty.
\end{defn}

\begin{lem}\label{20210707a}
Let $N$ be a semifree DG $B$-module with a semifree basis $\mathcal{B}=\{e_{\lambda}\}_{\lambda\in \Lambda}$. The following assertions hold.
\begin{enumerate}[\rm(1)]
\item For $D,D'\in\Diff_A^\delta(N)$, we have
$D=D'$ if and only if $D(e_\lambda)=D'(e_\lambda)$ for all $e_\lambda\in \mathcal{B}$.
In particular, $D^{\mathcal{B}}=D$ if and only if $D(e_\lambda)=0$ for all $e_\lambda\in \mathcal{B}$.
\item For all $D_1,D_2 \in \Diff_A^\delta(N)$, the mapping $D_1-D_2$ is $B^{\natural}$-linear.
\item For all $D\in\Diff_A^\delta(N)$ and all $f\in \grHom_{B^{\natural}}(N^{\natural},(N\otimes_BJ)^{\natural})$ we have $D+f\in \Diff_A^\delta(N)$.
\item We have the equality:
\begin{equation}\label{20210704-1}
\Diff_A^\delta(N)=D^{\mathcal{B}}+\grHom_{B^{\natural}}(N^{\natural},(N\otimes_BJ)^{\natural}).
\end{equation}
\end{enumerate}
\end{lem}

\begin{proof}
(1) If $D(e_\lambda)=D'(e_\lambda)$ for all $e_\lambda\in\mathcal{B}$,
then for every finite sum $\sum_{\lambda}e_\lambda b_\lambda\in N^{\natural}$ with $b_{\lambda}\in B^{\natural}$ we have the equalities
$
D\left(\sum_{\lambda}e_\lambda b_\lambda\right)
 =\sum_{\lambda}\left(D(e_\lambda)b_\lambda+e_\lambda \otimes\delta(b_\lambda)\right)
=\sum_{\lambda}\left(D'(e_\lambda)b_\lambda+e_\lambda \otimes\delta(b_\lambda)\right) =D'\left(\sum_{\lambda}e_\lambda b_\lambda\right)
$.
Hence, the equality $D=D'$ holds.

(2) For all $n\in N^{\natural}$ and $b\in B^{\natural}$ we have the equalities
$(D_1-D_2)(nb)= D_1(n)b+n\otimes\delta(b) - D_2(n)b -n\otimes\delta(b) = (D_1-D_2)(n)b$.
Hence, $D_1-D_2$ is $B^{\natural}$-linear.

(3) For all $n\in N^{\natural}$ and $b\in B^{\natural}$ we have
$
(D+f)(nb)
= D(nb)+f(nb)
= D(n)b+n\otimes\delta(b) + f(n)b
= (D+f)(n)b + n\otimes\delta(b)
$.
Therefore, $D+f\in\Diff_A^\delta(N)$.

(4) follows immediately from statements (2) and (3).
\end{proof}

\begin{notation}\label{para20210804d}
Let $N$ be a semifree DG $B$-module with a semifree basis $\mathcal{B}=\{e_{\lambda}\}_{\lambda\in \Lambda}$. Given a subset $\Gamma=\{\gamma_\lambda\}_{\lambda\in\Lambda}$ of $(N\otimes_BJ)^{\natural}$ with $|\gamma_\lambda|=|e_\lambda|$ for all $\lambda\in\Lambda$,
we define a graded $A^{\natural}$-linear homomorphism $D_\Gamma\colon N^{\natural}\to (N\otimes_BJ)^{\natural}$ of degree $0$ by
$$
D_\Gamma\left(\sum_{\lambda}e_{\lambda}b_{\lambda}\right)=\sum_{\lambda}\left(\gamma_{\lambda} b_{\lambda}+e_{\lambda}\otimes \delta(b_{\lambda})\right)
$$
where $b_{\lambda}\in B^{\natural}$ and $b_\lambda\neq 0$ for only finitely many $\lambda$. Note that $D_\Gamma$ is well-defined and $D_\Gamma\in\Diff_A^\delta(N)$ by definition.
\end{notation}

\begin{lem}\label{20210708a}
Let $N$ be a semifree DG $B$-module with a semifree basis $\mathcal{B}=\{e_{\lambda}\}_{\lambda\in \Lambda}$. There is a one-to-one correspondence
$$
\xymatrix{
\Diff_A^\delta(N)\   \ar@<0.7ex>[rr]^-{f}   && \ar@<.7mm>[ll]^-{g}\  {\displaystyle\prod_{\lambda\in\Lambda}(N\otimes_{B} J)^{\natural}_{|e_\lambda|}}
}
$$
defined by $f(D)=\{D(e_\lambda)\}_{\lambda\in\Lambda}$ and $g(\Gamma)=D_\Gamma$.
\end{lem}

\begin{proof}
Let $D\in\Diff_A^\delta(N)$ and $\Gamma=\{\gamma_\lambda\}_{\lambda\in\Lambda}\subset (N\otimes_BJ)^{\natural}$ such that $|\gamma_\lambda|=|e_\lambda|$ for $\lambda\in\Lambda$.
We have $(g f)(D)(e_\mu)=D_{\{D(e_\lambda)\}}(e_\mu)=D(e_\mu)$ for all $e_\mu\in\mathcal{B}$.
It follows from Lemma \ref{20210707a}(1) that $(gf)(D)=D$.
On the other hand, $(fg)(\Gamma)=\{D_\Gamma(e_\lambda)\}_{\lambda\in\Lambda}=\{\gamma_\lambda\}_{\lambda\in\Lambda}=\Gamma$.
\end{proof}

\begin{notation}\label{para20210804e}
Let $N$ be a semifree DG $B$-module with a semifree basis $\mathcal{B}=\{e_{\lambda}\}_{\lambda\in \Lambda}$. For every $D\in \Diff_A^\delta(N)$, let $$\alpha(D):= D\partial^N-\partial^{N\otimes_B J}D.$$ Note that $\alpha(D) \colon N^{\natural}\to (N\otimes_BJ)^{\natural}$ is a graded $A^{\natural}$-linear map of degree $-1$.
\end{notation}

\begin{lem}\label{20210706a}
Let $N$ be a semifree DG $B$-module with a semifree basis $\mathcal{B}=\{e_{\lambda}\}_{\lambda\in \Lambda}$, and let $D\in \Diff_A^\delta(N)$. Then the map $\alpha(D)\colon N\to \shift(N\otimes_B J)$ is a DG $B$-module homomorphism.
Hence, there is a mapping
$$
\Diff_A^\delta(N) \to \Hom_{\K(B)}(N,\shift  (N\otimes _{B} J) )
$$
defined by $D \mapsto [\alpha(D)]$.
Moreover, for all $D_1,D_2\in \Diff_A^\delta (N)$, the equality $$[\alpha(D_1)]=[\alpha(D_2)]$$ holds in $\Hom_{\K(B)}(N, \shift (N\otimes_B J))$.
\end{lem}

\begin{proof}
It is straightforward to see that $\partial^{\shift (N\otimes_B J)}\alpha(D) -\alpha(D) \partial^N = 0$. Hence, $\alpha(D)$ is a chain map.
On the other hand, for all $n\in N^{\natural}$ and $b\in B^{\natural}$ we have the equalities
\begin{align*}
\alpha(D)(nb)
=&  (D\partial^N-\partial^{N\otimes_B J}D ) (nb) \\
=& D (\partial^N(n)b+(-1)^{|n|}n d^B(b)) - \partial^{N\otimes_B J}(D(n)b+n\otimes \delta (b) ) \\
=& D(\partial^N(n))b + \partial^N(n)\otimes \delta(b) + (-1)^{|n|}D(n)d^B(b) + (-1)^{|n|}n\otimes \delta(d^B(b))\\
-& \partial^{N\otimes_B J}(D(n))b - (-1)^{|n|}D(n)d^B(b) - \partial^N(n)\otimes \delta(b)  - (-1)^{|n|} n \otimes \partial^J(\delta(b))\\
=&  (D\partial^N - \partial^{N\otimes_B J}D)(n) b\\
=& \alpha(D)(n) b
\end{align*}
where the fourth equality uses~\eqref{eq20210523s}. Hence, $\alpha(D)$ is $B^{\natural}$-linear and we conclude that $\alpha(D)\in \Hom_{B}(N, \shift (N\otimes_B J))$.
Note that $\alpha(D)$ defines an element $[\alpha(D)]$ in $\Hom_{\K(B)}(N, \shift (N\otimes_BJ))$.

For the last assertion, assume $D_1,\,D_2\in \Diff_A^\delta(N)$. Let $f=D_1-D_2$ and note that $f$ is $B^{\natural}$-linear by Lemma~\ref{20210707a}(2).  Then we have the equalities
\begin{align*}
\alpha(D_1)-\alpha(D_2)
& = ( D_1\partial^N - \partial^{N\otimes_B J}D_1 ) - ( D_2\partial^N - \partial^{N\otimes_B J}D_2 ) = \partial^{\shift N\otimes_B J}f + f\partial^N.
\end{align*}
Hence, $[\alpha(D_1)]=[\alpha(D_2)] $ in $\Hom_{\K(B)}(N, \shift (N\otimes_B J))$, as desired.
\end{proof}

\begin{defn}
Using the notation from Lemma~\ref{20210706a}, it follows that the class $[\alpha(D)]$ in $\Hom_{\K(B)}(N,\shift  (N\otimes _{B} J) )$ does not depend on the choice of $D\in \Diff_A^\delta(N)$.
We call $[\alpha(D)]$ the Atiyah class of $N$.
\end{defn}

\begin{disc}
There exist other definitions of Atiyah class in the literature; see, for instance,~\cite{BF,CM}.
However, our above definition of Atiyah class is different from the existing ones in~\cite{BF,CM}.
\end{disc}

In the following, recall that $\Delta_N$ is introduced in Theorem~\ref{prop20210802c}.

\begin{thm}\label{20210706d}
Let $N$ be a semifree DG $B$-module with a semifree basis $\mathcal{B}=\{e_{\lambda}\}_{\lambda\in \Lambda}$. The following equality holds:
\begin{equation}\label{eq20210805a}
\alpha(D^{\mathcal{B}})=\Delta_N.
\end{equation}
\end{thm}

\begin{proof}
For a basis element $\e_{\lambda}$ of $\mathcal{B}$, write $\partial^N(e_\lambda)=\sum_{\mu <\lambda}e_\mu b_{\mu\lambda}$, which is a finite sum with $b_{\mu\lambda}\in B^{\natural}$.
Then we have the equalities
\begin{align*}
\alpha(D^{\mathcal{B}})(e_\lambda)
 = (D^{\mathcal{B}}\partial^N- \partial^{N\otimes_B J} D^{\mathcal{B}})(e_\lambda)
 = D^{\mathcal{B}}\left(\partial^{N}(e_\lambda)\right)
 =  \sum_{\mu<\lambda}e_\mu \otimes \delta(b_{\mu\lambda})
 = \Delta_N(e_\lambda)
\end{align*}
in which the last equality follows from~\eqref{eq20210523b}.
Now, the equality~\eqref{eq20210805a} follows from the fact that $\alpha(D^{\mathcal{B}})$ and $\Delta_N$ are both $B^{\natural}$-linear.
\end{proof}


Moreover, we can prove the following result.

\begin{prop}\label{20210626a}
Let $N$ be a semifree DG $B$-module with a semifree basis $\mathcal{B}=\{e_{\lambda}\}_{\lambda\in \Lambda}$. The following statements are equivalent.
\begin{enumerate}[\rm(i)]
\item $N$ is na\"{i}vely liftable to $A$.
\item  $[\alpha(D)]=0$ holds in $\Ext_B^{1}(N, N\otimes_BJ)$ for all $D\in\Diff_A^\delta(N)$.
\item $[\alpha(D)]=0$ holds in $\Ext_B^{1}(N, N\otimes_BJ)$ for some $D\in\Diff_A^\delta(N)$.
\item $\alpha(D)=0$ holds in $\Hom_{B}(N,\, \shift (N\otimes_B J))$ for some $D\in\Diff_A^\delta(N)$.
\item $\alpha(D_\Gamma)=0$ holds in $\Hom_{B}(N,\, \shift (N\otimes_B J))$ for some subset $\Gamma=\{\gamma_\lambda\}_{\lambda\in\Lambda}$ of $(N\otimes_BJ)^{\natural}$ with $|\gamma_\lambda|=|e_\lambda|$ for all $\lambda\in \Lambda$.
\end{enumerate}
\end{prop}

\begin{proof}
(i)$\implies$(iii) and (ii)$\implies$(i) follow from Theorem~\ref{main} and Theorem~\ref{20210706d}.

(iii)$\implies$(ii) follows from Lemma \ref{20210706a}.

(iv)$\implies$(iii) is trivial.

(iii)$\implies$(iv) Let $D\in\Diff_A^\delta(N)$ such that $[\alpha(D)]=0$ in $\Ext_B^{1}(N, N\otimes_BJ)$. There is a graded $B^{\natural}$-module homomorphism $h\colon N^{\natural}\to (N\otimes_BJ)^{\natural}$ of degree 0 such that $\alpha(D)=\partial^{\shift N\otimes_B J}h+h\partial^N$.
Note that $D-h \in \Diff_A^\delta(N)$ by Lemma~\ref{20210707a}(3).
By definition we also have $\alpha(D)=D\partial^N - \partial^{N\otimes_B J}D$. Therefore,
$
\alpha(D-h)= (D-h)\partial^N - \partial^{N\otimes_B J}(D-h)=\alpha(D)-\alpha(D)=0.
$

(iv)$\Longleftrightarrow$(v) is clear from Lemma~\ref{20210708a}.
\end{proof}

\section{Some concrete examples}\label{sec20220306a}

In this section, we will construct concrete examples of DG modules that do and do not satisfy na\"{\i}ve liftability; see Examples~\ref{ex20210815a} and~\ref{ex20210815g}. The main tool to construct such examples is the following theorem, which is our main result in this section. (Again, our notation in this section comes from the previous sections and we still work in the setting of Convention~\ref{para20210317a}.)

\begin{thm}\label{20210626b}
Let $N$ be a semifree DG $B$-module with a semifree basis $\mathcal{B}=\{e_{\lambda}\}_{\lambda\in \Lambda}$. Write $\partial^N(e_\lambda)=\sum_{\mu<\lambda}e_{\mu}b_{\mu\lambda}$ as a finite sum with $b_{\mu\lambda}\in B^{\natural}$.
Then the following assertions are equivalent.
\begin{enumerate}[\rm(i)]
\item $N$ is na\"{i}vely liftable to $A$.
\item There is a subset $\{\gamma_\lambda\}_{\lambda\in\Lambda}$ of $(N\otimes_BJ)^{\natural}$ with
$|\gamma_\lambda|=|e_\lambda|$ for all $\lambda\in \Lambda$ such that
$\partial^{N\otimes_BJ} (\gamma_\lambda)=\sum_{\mu < \lambda}\left(\gamma_{\mu}b_{\mu\lambda}+e_{\mu}\otimes \delta(b_{\mu\lambda})\right)$.
\end{enumerate}
\end{thm}

We give the proof of Theorem~\ref{20210626b} after the following proposition.

\begin{prop}\label{disc20210806a}
Let $N$ be a semifree DG $B$-module with a semifree basis $\mathcal{B}=\{e_{\lambda}\}_{\lambda\in \Lambda}$, and let $\{\gamma_\lambda\}_{\lambda\in\Lambda}$ be a subset of $(N\otimes_BJ)^{\natural}$ with
$|\gamma_\lambda|=|e_\lambda|$ for all $\lambda\in \Lambda$. Then
$\sum_{\mu < \lambda}\left(\gamma_{\mu}b_{\mu\lambda}+e_{\mu}\otimes \delta(b_{\mu\lambda})\right)$ is a cycle in $N\otimes_B J$ for $\lambda\in\Lambda$.
\end{prop}

\begin{proof}
We proceed by induction on $\lambda$.
The assertion is trivial for $\lambda=\varpi$, where $\varpi$ denotes the minimum element of $\Lambda$.
Now, assume $\lambda>\varpi$. Use the general protocol that $\partial^N(e_i)=\sum_{j<i}e_{j}b_{ji}$ is a finite sum with all $b_{ji}\in B^{\natural}$. Then we have
{\small
\begin{align*}
&\partial^{N\otimes_B J}\left(\sum_{\mu < \lambda }\left(\gamma_{\mu}b_{\mu\lambda}+ e_{\mu}\otimes \delta(b_{\mu\lambda}\right))\right)\\
&=\sum_{\mu<\lambda}\left(\partial^{N\otimes_B J}(\gamma_{\mu})b_{\mu\lambda}+(-1)^{|\gamma_\mu|}\gamma_\mu d^B(b_{\mu\lambda})+\partial^N(e_{\mu})\otimes \delta(b_{\mu\lambda})+(-1)^{|e_{\mu|}}e_\mu \otimes \partial^J\delta(b_{\mu\lambda})\right)\\
&= \sum_{\mu < \lambda}  \left(\sum_{\nu < \mu} \left(\gamma_\nu b_{\nu\mu} + e_\nu\otimes \delta(b_{\nu\mu})\right)\right)b_{\mu\lambda} + \sum_{\mu < \lambda} (-1)^{|\gamma_\mu|}\gamma_\mu d^B(b_{\mu\lambda}) \\
& \qquad\qquad + \sum_{\mu < \lambda } \sum_{\nu < \mu }\left(e_{\nu}b_{\nu\mu}\otimes\delta(b_{\mu\lambda})\right)+\sum_{\mu < \lambda} \left((-1)^{|e_{\mu}|}e_\mu \otimes \partial^J\delta(b_{\mu\lambda})\right) \\
&=  \sum_{\nu<\lambda} \gamma_\nu \left( \sum_{\nu<\mu<\lambda}b_{\nu\mu}b_{\mu\lambda}+(-1)^{|\gamma_\nu|}d^B(b_{\nu\lambda}) \right)+ \sum_{\nu<\lambda} e_{\nu} \otimes \left( \sum_{\nu<\mu<\lambda}\delta(b_{\nu\mu} b_{\mu\lambda})+(-1)^{|e_\nu|}\delta(d^B(b_{\nu\lambda})) \right)\\
&= 0
\end{align*}}
where the second equality follows from the inductive hypothesis and the third equality follows from~\eqref{eq20210523s}.
To see the last equality, note that
\begin{align*}
0=(\partial^N)^2(e_{\lambda})=\partial^N\left(\sum_{\mu<\lambda}e_{\mu}b_{\mu\lambda}\right)
& =\sum_{\mu<\lambda}\left(\partial^N(e_{\mu})b_{\mu\lambda}+(-1)^{|e_{\mu}|}e_{\mu}d^Bb_{\mu\lambda}\right)\\
&=\sum_{\mu<\lambda}\left(\sum_{\nu<\mu}e_{\nu}b_{\nu\mu}b_{\mu\lambda}+(-1)^{|e_{\mu}|}e_{\mu}d^Bb_{\mu\lambda}\right)\\
&=\sum_{\nu<\lambda}e_{\nu}\left(\sum_{\nu<\mu<\lambda}b_{\nu\mu}b_{\mu\lambda}+(-1)^{|e_{\nu}|}d^Bb_{\nu\lambda}\right).
\end{align*}
Since $\mathcal{B}$ is a semifree basis for $N$, for $\nu<\lambda$, we conclude that
$$\sum_{\nu<\mu<\lambda}b_{\nu\mu}b_{\mu\lambda}+(-1)^{|e_{\nu}|}d^Bb_{\nu\lambda}=0$$
and hence, the last equality holds.
\end{proof}

\noindent {\it Proof of Theorem~\ref{20210626b}.}
Let $\Gamma=\{\gamma_\lambda\}_{\lambda\in\Lambda}$ be a subset of $(N\otimes_BJ)^{\natural}$ with $|\gamma_\lambda|=|e_\lambda|$ for $\lambda\in\Lambda$. Recall from Notation~\ref{para20210804d} that
the $A^{\natural}$-linear homomorphism $D_\Gamma\in \Diff_A^\delta(N)$ is defined by the formula
$D_\Gamma\left(\sum_{\lambda}e_{\lambda}b_{\lambda}\right)=\sum_{\lambda}\gamma_{\lambda} b_{\lambda}+e_{\lambda}\otimes \delta(b_{\lambda})$.
Now we have
\begin{align*}
\alpha(D_{\Gamma})(e_\lambda)
& = D_{\Gamma}(\partial^N(e_\lambda)) - \partial^{N\otimes_B J} (D_{\Gamma}(e_\lambda))
 = D_\Gamma\left(\sum_{\mu< \lambda}e_{\mu}b_{\mu\lambda}\right) - \partial^{N\otimes_B J}(\gamma_\lambda)\\
& = \sum_{\mu<\lambda}\left(\gamma_{\mu}b_{\mu\lambda}+e_{\mu}\otimes \delta(b_{\mu\lambda})\right) - \partial^{N\otimes_B J}(\gamma_\lambda).
\end{align*}
Using this equality, the assertion of Theorem~\ref{20210626b} follows from the equivalence (i)$\Longleftrightarrow$(v) in Proposition~\ref{20210626a}. \qed

\begin{disc}
If the elements $\gamma_\lambda$ of a set $\{\gamma_\lambda\}_{\lambda\in\Lambda}$ satisfy Theorem~\ref{20210626b}(ii), then they can be described by induction.
Let $\varpi$ denote the minimum element of $\Lambda$ and assume $\varpi<\lambda$.
Consider the element $\xi_\lambda=\sum_{\mu < \lambda}\left(\gamma_{\mu}b_{\mu\lambda}+e_{\mu}\otimes \delta(b_{\mu\lambda})\right)$ of $(N\otimes_BJ)^{\natural}$ which is constructed using the inductive step and note that it is a cycle in the DG $B$-module $N\otimes_B J$, by Proposition~\ref{disc20210806a}.
This $\xi_\lambda$ defines an element $[\xi_\lambda]$ in $\HH_{|e_\lambda|}(N\otimes_B J)$. By assumption, we obtain $\gamma_\lambda\in (N\otimes_BJ)^{\natural}$ satisfying $\partial^{N\otimes_BJ}(\gamma_\lambda)=\xi_\lambda$, that is, $[\xi_\lambda]=0$ in $\HH_{|e_\lambda|}(N\otimes_B J)$.
\end{disc}

The following result is an application of Theorem~\ref{20210626b} and will be used to construct Examples~\ref{ex20210815a} and~\ref{ex20210815g} below.

\begin{cor}\label{20210621}
Let $N$ be a semifree DG $B$-module with a semifree basis consisting of only two elements $\{e,\,e'\}$ with $\partial^N(e')=eb$ for $b\in B^{\natural}$.
Then the following hold:
\begin{enumerate}[\rm(a)]
\item If $\delta(b)$ is a boundary of $J$, then $N$ is na\"{i}vely liftable to $A$.
\item Conversely, if $A_0= B_0$ and $N$ is na\"{i}vely liftable to $A$, then $\delta(b)$ is a boundary of $J$.
\end{enumerate}
\end{cor}

\begin{proof}
(a) By assumption, there is $c\in J^{\natural}$ such that $\delta(b)=\partial^J(c)$.
Let $\gamma=0$ be in $(N\otimes_B J)_{|e|}$ and $\gamma'=e\otimes c$, where $\gamma' \in (N\otimes_B J)_{|e'|}$.
Note that $|\partial^N(e)|=|e|-1$. Hence, $\partial^N(e)=0$ and we have
$$\partial^{N\otimes_B J}(\gamma')= \partial^N(e)\otimes c + e \otimes \partial^J(c)=e \otimes \delta(b)=\gamma b +e \otimes \delta(b).$$
This means that $\{\gamma,\gamma'\}$ satisfies condition (ii) in Theorem~\ref{20210626b}.
Hence, $N$ is na\"{i}vely liftable to $A$.

(b) The assumption $B_0=A_0$ implies $J_0=0$.
Then, we have $(N\otimes_B J)_n=\bigoplus_{i>0}N_{n-i}\otimes J_i$.
In particular, $(N\otimes_B J)_{|e|}=0$.
Since $N$ is na\"{i}vely liftable to $A$, it follows from Theorem~\ref{20210626b} that there are elements
$\gamma=0$ in $(N\otimes_B J)_{|e|}$ and $\gamma'\in(N\otimes_B J)_{|e'|}$ such that
$\partial^{N\otimes_B J}(\gamma')=\gamma b +e \otimes \delta(b)=e \otimes \delta(b)$.
We know that $\gamma'=e\otimes c$, for some $c\in J$. Therefore,
$\partial^{N\otimes_B J}(\gamma')=\partial^{N\otimes_B J}(e\otimes c)=e\otimes \partial^J(c)$.
Hence, $\delta(b)=d^J(c)$ and this means that $\delta(b)$ is a boundary of $J$.
\end{proof}

\begin{cor}\label{cor20210815a}
Consider the assumption of Corollary~\ref{20210621}.
If $\HH_{|e'|-|e|-1}(J)=0$, then $N$ is na\"{i}vely liftable to $A$.
\end{cor}

\begin{proof}
Note that $0=(\partial^N)^2(e')=\partial^N(eb)=(-1)^{|e|}ed^B(b)$. This implies that $d^B(b)=0$. It follows from~\eqref{eq20210523s} that $\partial^J\delta(b)=\delta d^B(b)=0$, Hence, $\delta(b)$ is always a cycle in $J$. By our assumption, $\delta(b)$ is a boundary of $J$ as well. Now, it follows from Corollary~\ref{20210621}(a) that $N$ is na\"{i}vely liftable to $A$.
\end{proof}


Next, we construct a DG module that satisfies na\"{i}ve liftability. For the notation and more details about free extensions of DG algebras see~\cite{NOY2}.

\begin{ex}\label{ex20210815a}
Assume $x,y$ are non-zero elements of $R$ such that $xR\cap yR = (0)$. Let $B=R\langle X,Y\mid dX=x, dY=Xy\rangle$ be a free extension of the DG $R$-algebra $R$ with $|X|=1$ and $|Y|=2$. Let $N$ be the semifree DG $B$-module with $N^{\natural}=eB^{\natural}\oplus e'B^{\natural}$ and with the differential defined by $\partial^N(e)=0$ and $\partial^N(e')=eYXy$. By~\eqref{eq20210523s} we have $\delta(YXy)=\delta (d^B(Y^{(2)}))=\partial^J\delta(Y^{(2)})$. Now it follows from Corollary~\ref{20210621}(a) that $N$ is na\"{i}vely liftable to $R$.
\end{ex}

In the next example we construct a DG module that does not satisfy na\"{i}ve liftability.

\begin{ex}\label{ex20210815g}
Consider the setting of Example~\ref{ex20210815a}. Let $M$ be the semifree DG $B$-module with $M^{\natural}=uB^{\natural}\oplus u'B^{\natural}$ and with the differential defined by $\partial^{M}(u)=0$ and $\partial^M(u')=uYXx$, where $|u|=0$.
\vspace{5pt}

\noindent \textbf{Claim}: $M$ is not na\"{i}vely liftable to $R$.
\vspace{5pt}

To prove the claim, set
\begin{align*}
T&=B\langle X', Y' \mid dX'=0, dY'= X'y \rangle\\
&=R\langle X, X', Y, Y' \mid dX=x, dX'=0, dY = Xy, dY'= X'y \rangle
\end{align*}
where $|X'|=|X|=1$ and $|Y'|=|Y|=2$.
Note that
$$
T^{\natural}=\bigoplus_{n\geq 0} \left(Y'^{(n-1)}X'B^{\natural}\oplus Y'^{(n)}B^{\natural}\right)
$$
with the convention $Y^{(-1)}=0$.
The map
$f\colon B^e \to T$
defined by $f(1\otimes X)=X,\,f(X\otimes 1)=X'+X,\,f(1\otimes Y^{(n)})=Y^{(n)},\,f(Y^{(n)}\otimes 1)=(Y'+Y)^{(n)}$ is a DG $R$-algebra isomorphism.
In particular, if we consider the DG ideal $J'$ of $T$, where $J'^{\natural}=\bigoplus_{n> 0} \left(Y'^{(n-1)}X'B^{\natural}\oplus Y'^{(n)}B^{\natural}\right)$ and $\partial^{J'}$ is induced by $d^T$, then we obtain an isomorphism
$f|_J\colon J  \xra{\cong}  J'$ which induces a mapping $\HH(f|_J)\colon \HH(J)\to \HH(J')$ on homology.
Note that $f(\delta(YXx))=X'Yx+Y'Xx+X'Y'x$. Since $\partial^{J'}(-X'XY)=X'Yx$ and $\partial^{J'}(-X'Y'X)=X'Y'x$,
we have the following equalities:
$$\HH(f|_J)([\delta(YXx)])=[X'Yx+Y'Xx+X'Y'x]=[Y'Xx].$$
We show that $[Y'Xx]\neq 0$ in $\HH_3(J')$.
To see this, note that for a general element $X'XYa+Y'Yb\in(X'B\oplus Y'B)_4=X'XYR\oplus Y'YR$ with $a,b\in R$ we have
$$
\partial^{J'}(X'XYa+Y'Yb)=X'Y(-xa+yb)+Y'Xyb.
$$
If $[Y'Xx]=0$, then there is $b\in R$ such that $yb=x$.
It then follows from the assumption $xR\cap yR=0$ that $x=0$.
This contradicts the fact that $x$ is nonzero.
Hence, $[Y'Xx]\neq 0$, that is, $\delta(YXx)$ is not boundary in $J$.
It then follows from Corollary~\ref{20210621}(b) that $M$ is not na\"{i}vely liftable to $R$, as desired.
\end{ex}




\begin{thebibliography}{10}


\bibitem{auslander:lawlom}
M.~Auslander, S.\ Ding, and \O.\ Solberg, \emph{Liftings and weak liftings of
  modules}, J. Algebra \textbf{156} (1993), 273--397.

\bibitem{avramov:ifr}
L.~L. Avramov, \emph{Infinite free resolutions}, Six lectures on commutative
  algebra (Bellaterra, 1996), Progr. Math., vol. 166, Birkh\"auser, Basel,
  1998, pp.~1--118.

\bibitem{avramov:dgha}
L.~L. Avramov, H.-B.\ Foxby, and S.\ Halperin, \emph{Differential graded
  homological algebra}, in preparation.

\bibitem{BF}
R.~Buchweitz and H.~Flenner,
\emph{A semiregularity map for modules and applications to deformations}, Compositio Math. \textbf{137} (2003), no. 2, 135--210.

\bibitem{CM}
N.~Carqueville and D.~Murfet,
\emph{Adjunctions and defects in Landau-Ginzburg models}, Adv. Math. \textbf{289} (2016), 480--566.

\bibitem{AC}
A.~Connes,
{\it Noncommutative differential geometry},
Inst. Hautes \'{E}tudes Sci. Publ. Math. No. 62 (1985), 257--360.

\bibitem{CQ}
J.~Cuntz and D.~Quillen,
{\it Algebra extensions and nonsingularity},
J. Amer. Math. Soc. 8 (1995), no. 2, 251--289.


\bibitem{felix:rht}
Y.\ F{\'e}lix, S.\ Halperin, and J.-C.\ Thomas, \emph{Rational homotopy
  theory}, Graduate Texts in Mathematics, vol. 205, Springer-Verlag, New York,
  2001.

\bibitem{GL}
Tor H.~Gulliksen and G.~Levin,
{\it Homology of local rings}, Queen's Paper in Pure and Applied Mathematics, No. 20 (1969), Queen's University, Kingston, Ontario, Canada.


\bibitem{NOY}
S.~Nasseh, M.~Ono, and Y.~Yoshino, \emph{The theory of $j$-operators with application to (weak) liftings of DG modules}, J. Algebra, {\bf 605} (2022), 199--225.

\bibitem{NOY2}
S.~Nasseh, M.~Ono, and Y.~Yoshino, \emph{Na\"{\i}ve liftings of DG modules}, Math. Z., {\bf 301} (2022), no. 1, 1191--1210.


\bibitem{nasseh:lql}
S.~Nasseh and S.~Sather-Wagstaff, \emph{Liftings and Quasi-Liftings of DG modules}, J. Algebra, {\bf 373} (2013), 162--182.

\bibitem{nassehyoshino}
S.~Nasseh and Y.~Yoshino, \emph{Weak liftings of DG modules}, J. Algebra, {\bf 502} (2018), 233--248.

\bibitem{OY}
M.~Ono and Y.~Yoshino, \emph{A lifting problem for DG modules}, J. Algebra {\bf 566} (2021), 342--360.

\bibitem{yoshino}
Y.\ Yoshino, \emph{The theory of {L}-complexes and weak liftings of complexes},
  J. Algebra \textbf{188} (1997), no.~1, 144--183.

\end{thebibliography}
\providecommand{\bysame}{\leavevmode\hbox to3em{\hrulefill}\thinspace}
\providecommand{\MR}{\relax\ifhmode\unskip\space\fi MR }
\providecommand{\MRhref}[2]{%
  \href{http://www.ams.org/mathscinet-getitem?mr=#1}{#2}
}
\providecommand{\href}[2]{#2}

\end{document}